\documentclass{amsart}

\usepackage{cite}
\usepackage{amscd,amssymb,amsmath,nccmath,microtype}
\usepackage{calligra,mathrsfs}
\usepackage{graphicx}
\usepackage[hyphens]{url}
\usepackage{hyperref}
\usepackage{booktabs}
\usepackage{changepage}
\usepackage{caption}
\usepackage{listings}
\usepackage{amssymb}
\usepackage{amsfonts}
\usepackage{amsmath}
\usepackage{nicefrac}
\usepackage{stmaryrd}

\usepackage[hyphens]{url}
\usepackage{hyperref}
\usepackage{lmodern}
\usepackage{xcolor}
\usepackage{colortbl}
\usepackage[most]{tcolorbox}
\usepackage{booktabs}
\usepackage{longtable}

\usepackage{xypic}

\usepackage{tikz}
\usepackage{multirow}
\usepackage{hyperref}
\usepackage{subcaption} 
\usepackage{mathtools}
\usepackage[draft,multiuser,layout={margin,index}]{fixme}
\usepackage{pifont}

\usepackage{tikz-cd}

\makeatletter
\newcommand*{\@rowstyle}{}

\newcommand*{\rowstyle}[1]{
  \gdef\@rowstyle{#1}%
  \@rowstyle\ignorespaces%
}

\newcolumntype{=}{
  >{\gdef\@rowstyle{}}%
}

\newcolumntype{+}{
  >{\@rowstyle}%
}

\makeatother

\newtheorem{theorem}{Theorem}[section]
\newtheorem*{ttheorem}{Tian's criterion}
\newtheorem{lemma}[theorem]{Lemma}
\newtheorem{proposition}[theorem]{Proposition}

\newtheorem{corollary}[theorem]{Corollary}
{\theoremstyle{remark}

}

\theoremstyle{definition}
\newtheorem{construction}[theorem]{Construction}
\newtheorem{definition}[theorem]{Definition}
\newtheorem{example}[theorem]{Example}

\newcommand{\CC}{\mathbb{C}}

\newcommand{\RR}{\mathbb{R}}

\newcommand{\QQ}{\mathbb{Q}}

\newcommand{\PP}{\mathbb{P}}

\DeclareMathOperator{\Bl}{Bl}

\DeclareMathOperator{\spec}{Spec}
\DeclareMathOperator{\proj}{Proj}

\DeclareMathOperator{\Aut}{Aut}
\DeclareMathOperator{\glct}{g\bf{lct}}

\DeclareMathOperator{\lct}{lct}

\title[K\"ahler-Einstein metrics on symmetric general arrangement varieties]{K\"ahler-Einstein metrics on symmetric general arrangement varieties}

\author[J. Cable]{Jacob Cable}
\address{Jacob Cable\\ School of Mathematics, Faculty of Science and Engineering,
The University of Manchester,
Alan Turing Building, Oxford Road,
Manchester M13 9PL}
\email{\href{mailto:jacob.cable@manchester.ac.uk}{jacob.cable@manchester.ac.uk}}


\begin{document}
\begin{abstract}
  We calculate Chow quotients of some families of symmetric \(T\)-varieties. In complexity two we obtain new examples of K\"ahler-Einstein metrics by bounding the symmetric alpha invariant of their orbifold quotients. As an additional application we determine the homeomorphism class of the orbit space of the compact torus action.
\end{abstract}
\maketitle
\section{Introduction}

In this paper we are interested in finding new K\"ahler-Einstein metrics on certain Fano manifolds. Whether or not a manifold \(X\) admits such a metric is characterized by the algebreo-geometric notion of \(K\)-stability, see \cite{CDS12}, \cite{CDS13}. An equivariant version of \(K\)-stability has been used in the spherical \cite{del2016}, and complexity-one \(T\)-variety \cite{ilten2015} settings to obtain an effective K\"ahler-Einstein criterion.

Generally, with group actions of complexity greater than one, \(K\)-stability is a difficult condition to check. Another approach is to apply the sufficient criterion of Tian, \cite{Tian87}. In \cite{Su13} it was shown that Tian's criterion reduces to a related problem on the torus quotient if the original variety admits enough additional symmetries.

In the present paper the examples we consider are \textit{general arrangement varieties}, that is they are \(T\)-varieties where the torus quotient is a projective space and the critical values of the quotient map form a general arrangement of hyperplanes in that projective space. Smooth projective general arrangement varieties of complexity and Picard rank \(2\) were classified according to their Cox ring in \cite{hausen2018torus}. Following the methods of \cite{Su13}, we find three new examples of K\"ahler-Einstein metrics on some symmetric complexity two general arrangement varieties.

The first examples we are interested in are some hypersurfaces of  bidegree \((\alpha,\beta)\). Consider the following varieties
\[
X_{\alpha,\beta}^{2n-1} := V \left( \sum_{i=0}^n x_i^\alpha y_i^\beta \right) \subseteq \PP^n \times \PP^n
\]
Let \(d = \gcd(\alpha,\beta)\) and \(a= \alpha/d, \ b = \beta/d\). There is an effective \(n\)-torus action prescribed by weights  \((0|bI_n|0|-aI_n)\) on the homogeneous coordinates \((x,y)\). Note, as an aside, that \(X^{2n-1}_{1,1}\) is a flag manifold of type \((1,n-1)\), which is known to admit a K\"ahler-Einstein metric as a homogeneous manifold, see remarks immediately preceeeding \cite[Theorem 3]{Matsushima} for example. By adjunction \(X^{2n-1}_{\alpha,\beta}\) is Fano iff \(\alpha,\beta < n+1\). When \(\alpha = 1\) then \(X^{2n-1}_{\alpha,\beta}\) is smooth. Our first result is the following:
\begin{theorem}\label{thm:KE1}
\(X_{1,2}^5\) and \( X_{1,3}^5\) admit \(T\)-invariant K\"ahler-Einstein metrics.
\end{theorem}
Note that \(X_{1,2}^5, X_{1,3}^5\ \) appear in the classification of \cite{hausen2018torus} as varieties \(4E, 4F\) respectively. Secondly we consider a blowup of an even dimensional quadric hypersurface. Consider the following representation:
\begin{align*}
Q^{2n} &:= V \left( \sum_{i=0}^{n} x_{2i}x_{2i+1} \right) \subset \PP^{2n+1} \\
\end{align*}
There is a \(T = (\CC^*)^{n+1}\)-action, given by specifying \(\deg x_{2i} = e_{i+1}, \ \deg x_{2i+1} = -e_{i+1}\) for \(i=0,\dots,n\). This action is not effective, but we may quotient by the global stabilizer to obtain an effective action of the torus \(T' = T/(\pm \text{Id})\).

Let \(Z_i = V(x_{2i},x_{2i+1}) \subset Q^{2n}\) for \(i = 0,\dots, n\). Let \(W^{2n}\) be the variety obtained by iteratively blowing up the strict transforms of these varieties:
\[
W^{2n} := \Bl_{\tilde{Z}_n} \Bl_{\tilde{Z}_{n-1}} \dots \Bl_{\tilde{Z}_2} \Bl_{Z_1} Q^{2n}
\]
In Section \ref{subsec:wonderful} we show that \(W^{2n}\) is Fano. Our second result is the following:
\begin{theorem}\label{thm:KE2}
\(W^6\) admits a \(T'\)-invariant K\"ahler-Einstein metric.
\end{theorem}
We will use the following theorem of Tian, involving an invariant \(\alpha_{G}(X)\), the alpha invariant of a Fano manifold \(X\).
\begin{ttheorem}\label{thm:tcrit}
Let \(X\) be a Fano manifold and \(G \subset \Aut(X)\) reductive group of symmetries. If
\[
\alpha_{G} (X) > \frac{\dim(X)}{\dim(X) + 1}
\]
Then \(X\) admits a \(G\)-invariant K\"ahler-Einstein metric.
\end{ttheorem}
In the case where \(G\) is finite, the alpha invariant coincides with the \textit{global log canonical threshold} of \(X\), denoted \(\glct_G(X)\). A proof of this fact can be found in Demailley's appendix of \cite{cheltsov08}. 
Suppose now \(X\) admits an effective action of an algebraic torus \(T = (\CC^*)^m \subset \Aut(X)\). Let \(K \cong (S^1)^n\) be the maximal compact subtorus in \(T\). Given a finite subgroup \(H\) of the normalizer \(\mathcal{N}_{\Aut(X)}(T)\), we have an induced \(H\)-action on the character lattice \(M\) of \(T\).

Following \cite{batyrev99}, \(X\) is said to be symmetric with respect to the \(T\)-action if there exists such an \(H\) which fixes only the identity of \(M\). Since \(H\) normalizes \(T\) then they generate a subgroup \(H T \subset \Aut(X)\). As \(H\) and \(T\) have trivial intersection this is a semidirect product of \(H\) and \(T\). It can be shown that \(HT\) is the complexification of its maximal compact subgroup \(HK\). Minor adjustments to Demailly's proof give us that \( \glct_{HT}(X) = \alpha_{HT}(X) \) here also.

Let \(X\) be a symmetric \(T\)-variety. There is a universal quotient \(\pi:X \dashrightarrow Y\) by the torus action, the Chow quotient, first introduced by Kapranov \cite{kapranov1993}. Since \(H\) is in the normalizer of \(T\) then its action descends to \(Y\). In the setting of this paper we assume that \(\pi\) is surjective. For any prime divisor \(Z\) on \(Y\) the generic stabilizer on a component of \(\pi^{-1}(Z)\) is a finite abelian group. The maximal order across these components is denoted \(m_Z\). We then obtain a boundary divisor for \(\pi\) given by:
\begin{equation} \label{boundary}
B := \sum_Z \frac{m_Z-1}{m_Z} \cdot Z
\end{equation}
We call the pair \((Y,B)\) the Chow quotient pair of the \(T\)-variety \(X\). S{\"u}{\ss}, \cite{Su13}, showed that the global log canonical threshold of \(X\) with respect to \(HT\) coincides with that of the pair \((Y,B)\) with respect to \(H\):
\begin{theorem}[S{\"u}{\ss}, {\cite[Theorem 1.2]{Su13}}]\label{thm:SU}
Let \(X\) be a symmetric log terminal Fano \(T\)-variety. Then
\[
\glct_{HT}(X) = \min \{1, \glct_H(Y,B) \}.
\]
\end{theorem}
To prove Theorems \ref{thm:KE1}, \ref{thm:KE2} we would like to apply Theorem~\ref{thm:SU} and then Tian's criterion. We begin by calculating Chow quotients. The GIT quotients of \(X\) by \(T\) form a finite inverse system, and the inverse limit of this system contains a distinguished component, known as the \textit{limit quotient} of \(X\). In \cite{baker2012} it was shown that this limit quotient and the Chow quotient coincide. For the varieties \(X_{\alpha,\beta}^{2n-1}\) we use the Kempf-Ness theorem to calculate GIT quotients. The inverse system is simple enough in this case to then deduce the Chow quotient pair.

For \(\gamma \in \QQ\) define the \(\QQ\)-divisor \(B_\gamma := \gamma \sum_i H_i\), where \(H_1,\dots,H_n\) are the coordinate hyperplanes of \(\PP^{n-1}\) and \(H_0\) is the hyperplane \(V( \sum_{i=1}^n z_i)\). In Section~\ref{subsec:hypersurfaces} we will prove the following:
\begin{lemma}\label{lem:1.4}
The Chow quotient pair of \(X_{\alpha,\beta}^{2n-1}\) by \(T\) is \((\PP^{n-1},B_\gamma)\) with \(\gamma = \max \left(\frac{a-1}{a}, \frac{b - 1}{b} \right)\).
\end{lemma}
By \cite[Proposition 2.6]{suess18-2} this allows us to calculate, as an additional application, the homeomorphism class of the quotient space \(X_{\alpha,\beta}^n/T\). For the proof of the following corollary see Section~\ref{subsec:hypersurfaces}.
\begin{corollary}\label{cor:topquot}
Let \(K\) be the maximal compact torus of \(T\). There is a homeomorphism:
\[
X_{\alpha,\beta}^{2n-1}/K \cong S^{n-1} \ast \PP^{n-1}.
\]
Where the later is the topological join of the \((n-1)\)-sphere and complex projective \((n-1)\)-space. 
\end{corollary}
In particular, this shows that the \(K\)-orbit space of the flag manifolds \(F(1,n-1,\CC^n) = X_{1,1}^{2n-1}\) is of this form.

Consider the variety \(W^{2n}\). Using results of \cite{kirwan} we may obtain the Chow quotient of \(W^{2n}\) from that of \(Q^{2n}\). In Section \ref{subsec:wonderful} we prove the following:
\begin{lemma}\label{lem:1.6}
The Chow quotient pair of \(W^{2n}\) by its \(T'\)-action is \((\PP^{n-1},B_{\nicefrac{1}{2}})\).
\end{lemma}
Note there is a natural \(S_{n+1}\)-action on \(X^{2n-1}_{\alpha,\beta}\) permuting the indices of variables. Additionally, by results of \cite{li06}, the \(S_{n+1}\)-action on \(Q^{2n}\) permuting the \(Z_i\) induces an action on \(W^{2n}\). These actions descend to the \(S_{n+1}\)-action on \(\PP^{n-1}\) permuting the hyperplanes \(H_i\).

We are now in a situation where we may be able to apply Theorem~\ref{thm:SU}. For \(n=3\) we are able to provide the following lower bound on the global log canonical threshold of the pair \((\PP^2,B_\gamma)\), by considering degenerations under a \(\CC^*\)-action. 
\begin{lemma}\label{lem:alph}
Consider a log pair \((\PP^2,B_\gamma)\), where \(B_\gamma = \gamma \sum_i H_i\). We then have:
\[
\glct_{S_4}(\PP^2, B_\gamma) \ge
\begin{cases}
\frac{1}{3-4 \gamma} & \text{for }  \gamma \le \frac{1}{2};\\
\frac{2(1-\gamma)}{3-4\gamma}, & \text{for } \frac{1}{2} \le \gamma \le \frac{3}{4};\\
\infty, & \text{for } \gamma \ge \frac{3}{4}.
\end{cases}
\]
\end{lemma}
\subsection*{Acknowledgement}
This work was partially supported by the grant 346300 for IMPAN from the Simons Foundation and the matching 2015-2019 Polish MNiSW fund. I would also like to thank my supervisor Hendrik S{\"u}{\ss} for many useful conversations while this paper was in preparation.
\section{Preliminaries}
\subsection{Chow and GIT quotients} \
Here we recall the definition of GIT, Chow, and limit quotients of a projective variety by a reductive algebraic group \(G\). We also explain how, when \(G\) is a torus, they may be explicitely calculated via the Kempf-Ness theorem.
\subsubsection{GIT quotients}
Recall the basic setup of Mumford's geometric invariant theory, which provides a method for finding geometric quotients on open subsets of a scheme \(X\) when the acting algebraic group \(G\) is reductive. A good reference for the material here is \cite{mumford1994}.
In \cite{mumford1994} Mumford introduced the notion of a good categorical quotient, which can be shown to be unique if it exists.
\begin{definition}
A surjective \(G\)-equivariant morphism \(\pi : X \to Y\) is a good categorical quotient if the following hold:
\begin{enumerate}
\item We have \(\mathcal{O}_Y = (\pi_* \mathcal{O}_X)^G\);
\item if \(V\) is a closed \(G\)-invariant subset of \(X\) then \(\pi(V)\) is closed;
\item if \(V,W\) are closed \(G\)-invariant subsets of \(X\) and \(V \cap W = \emptyset\) then we have \(\pi(V) \cap \pi(W) = \emptyset\).
\end{enumerate}
\end{definition}
Good quotients do not always exist for a given scheme \(X\), but we might hope that there exists some dense open subset of \(X\) which does admit a good quotient. Consider the affine case, where \(X = \spec A\). For \(G\) reductive then it can be shown that \(X \sslash G : = \spec A^G\) is a good categorical quotient.

The same ansatz works in the projective case once we make a choice of a lift of the action to the ring of sections of a given ample line bundle. This choice is known as a linearization of the group action.
\begin{definition}
 Let \(X\) be a projective scheme together with an action \( \lambda : G \times X \to X\) of a reductive algebraic group \(G\). Let \(L\) be a line bundle on \(X\). A linearization of the action \(\lambda\) on \(L\) is an action \(\tilde{\lambda}\) on \(L\) such that:
\begin{itemize}
\item The projection \(\pi\) is \(G\)-equivariant, \(\pi \circ \tilde{\lambda} = \lambda \circ \pi \)
\item For \(g \in G\) and \(x \in X\), the induced map \(L_x \mapsto L_{g \cdot x}\) is linear.
\end{itemize}
\end{definition}
Note a linearization to \(L\) naturally induces linearizations to \(L^\vee\) and \(L^{\otimes r}\) for \(r \in \mathbb{N}\).
\begin{example}
A linearization of the trivial bundle on a projective variety \(X\) must be of the form:
\[
g \cdot (x,z) = (g \cdot x, \chi(g,x)z),
\]
for some \(\chi \in H^0(G \times X, \mathcal{O}_{G \times X}^*) \cong H^0(G, \mathcal{O}_G^*) = \mathfrak{X}(G).\)
\end{example}
The above example tells us that any two linearizations \(\lambda_1,\lambda_2\) of an action to the same line bundle differ by multiplication by some character \(\chi\) of \(G\): fiberwise we have \(\tilde{\lambda}_1 = \chi(g,x) \tilde{\lambda}_2\).

A linearization \(u\) of a group action \(G\) on \(X\) to \(L\) induces an action of \(G\) on the ring of sections \(R(X,L) := \bigoplus_{j \ge 0} H^0(X,L^{\otimes j}) \). Consider the scheme \(X \sslash_u G := \proj R(X,L)^G\). Note we have a birational map from \(X\) to \(X \sslash_u G\), defined precisely at those \(x \in X\) such that there exists some \(m> 0 \) and \(s \in R(X,L)^G_m\) such that \(s(x) \neq 0\). Such a \(x\) are said to be semi-stable. If in addition \(G \cdot x\) is closed and the stabilizer \(G_x\) is of dimension zero, the point \(x\) is said to be stable. The sets of semi-stable and stable points will be denoted by \(X^{ss}(u)\) and \(X^{s}(u)\) respectively.
\begin{construction}[{\cite[Chapter 1, Section 4]{mumford1994}}]
The canonical morphism \(X^{ss}(u) \to X\sslash_u G := \proj R(X,L)^G \) is a good categorical quotient.
\end{construction}
\subsubsection{Chow and limit quotients}
Recall the definition of the Chow quotient, as introduced in \cite{kapranov1993}. If \(G\) is any connected linear algebraic group and \(X\) is a projective \(G\)-variety, then orbit closures of points are generically of the same dimension and degree, and so define points in the corresponding Chow variety. The Chow quotient of the \(G\)-action on \(X\) is the closure of this set of points.

We now recall the definition of the limit quotient, from \cite{mumford1994}. The limit quotient is discussed in detail in \cite{baker2012}. Let \(G\) be a reductive algebraic group, and \(X\) a projective \(G\)-variety. Suppose there are finitely many sets of semi-stable points \(X_1,\dots,X_r\) arising from \(G\)-linearized ample line bundles on \(X\).  Whenever \(X_i \subseteq X_j\) holds, there is a dominant projective morphism \(X_i \sslash G \to X_j \sslash G\) which turns the set of GIT quotients into an inverse system. The associated inverse limit \(Y\) admits a canonical morphism \(\bigcap_{i=1}^r X_i \to Y\). The closure of the image of this morphism is the limit quotient.

When \(G\) is an algebraic torus there are indeed finitely many semi-stable loci. Moreover, by \cite[Corollary 2.7]{baker2012}, we may calculate the limit quotient by taking the inverse limit of the subsystem obtained by only considering linearizations of powers of one fixed ample line bundle \(L\).
In \cite[Proposition 2.5]{baker2012} it is shown that the Chow quotient and limit quotient coincide when \(G\) is an algebraic torus.
\subsubsection{Kempf-Ness approach to GIT quotients}
One approach to calculating GIT quotients is via the Kempf-Ness theorem. Let \(X \subseteq \PP^N\) be a nonsingular complex projective variety and let \(G\) be reductive algebraic group acting effectively on \(\PP^N\), restricting to an action on \(X\). Let \(K\) be the maximal compact subgroup in \(G\), with Lie algebra \(\mathfrak{k}\). The action of \(G\) is given by a representation \(\rho: G \to \text{GL}(N+1)\), and by choosing appropriate coordinates we may assume \(K\) maps to \(U(N+1)\) and so preserves the Fubini-Study form. It can be checked that a moment map \(\mu: X \to \mathfrak{k}^*\) is given by:
\begin{equation}\label{eq:mu}
\mu([x]) \cdot a := \frac{x^t \rho_*(a) x}{ |x|^2},
\end{equation}
where \(x\) is any representative of \([x] \in X \subseteq \PP^N\). Note we are now in the situation of the previous subsection, with \(L = \mathcal{O}_X(1)\) under the embedding \(X \subseteq \PP^N\). This moment map is unique up to translations in \(\mathfrak{k}^*\). A different choice of linearization in this setting corresponds to multiplying \(\rho\) by some character \(\chi \in \mathfrak{X}(G)\).

Since \(\chi(K)\) is compact it sits inside \(S^1 \subset \CC^*\), and hence we do not need to change coordinates when considering the effect on the moment map. When we plug this into (\ref{eq:mu}) we see that we have translated the moment map by \(\chi \in \mathfrak{X}(G) \otimes \RR \cong \mathfrak{k}^*\). Moreover, taking the \(r\)th power of \(L\) corresponds to scaling the moment map by a factor of \(r\). This gives a correspondence between rational elements \(\chi \in \mathfrak{X}(G) \otimes \QQ \subset \mathfrak{k}^*\) and linearizations of powers of \(L\).
\begin{example}
Suppose \(G = T\) is an algebraic torus with character and cocharacter lattices \(M,N\) respectively. Then \(\rho\) is a diagonal matrix of characters \(u_0,\dots,u_{N}\) and we obtain:
\[
\mu([x]) = \frac{\sum_{j=0}^N |x_i|^2 u_i}{|x|^2} \in M.
\]
Then, by Atiyah, \cite{atiyah1982convexity}, and Guillemin-Sternberg, \cite{guillemin1982convexity}, the image of \(\mu\) is a convex polytope \(P \subset M\).
\end{example}
We will make use of the following theorem of Kempf and Ness. A proof is given in \cite[Chapter 8]{mumford1994}. See also the original work \cite{kempf1979}.
\begin{theorem}[{\cite[Theorem 8.3]{kempf1979}}]\label{thm:KN}
Let \(X \subseteq \PP^N\) be a nonsingular complex projective variety and let \(G\) be reductive algebraic group acting effectively on \(\PP^N\), restricting to an action on \(X\). Consider a linearization of some power of \(L\) corresponding to a rational element \(u \in \mathfrak{k}^*\).
\begin{enumerate}
\item \(X^{ss}(u) = \{ x \in X | \overline{Gx} \cap \mu^{-1}(u) \neq \emptyset \} \). \\
\item The inclusion of \(\mu^{-1}(u)\) into \(X^{ss}(u)\) induces a homeomorphism
\[
\mu^{-1}(u)/K \to X\sslash_u G,
\]
where \(\mu^{-1}(u)/K\) is endowed with the quotient topology induced from the classical (closed submanifold topology) on \(\mu^{-1}(u)\), and \(X \sslash_u G\) is endowed with its classical (complex manifold) topology
\end{enumerate}
\end{theorem}
We can use Theorem~\ref{thm:KN} to calculate GIT quotients by inspection. To be explicit, suppose \(\mu^{-1}(u)/K\) has the structure of a complex projective variety and \(q: X^{ss}(u) \to \mu^{-1}(u)/K\) is a \(G\)-invariant morphism which restricts to the topological quotient map on the moment fibre, such that \(q_* \mathcal{O}_X^G = \mathcal{O}_Y\). The following fact is probably well known, but we prove it here for the reader's convenience.
\begin{lemma}\label{lem:catquot}
The morphism \(q\) is a good categorical quotient and hence is isomorphic to the GIT quotient map \(X \to X\sslash_u G\).
\end{lemma}
\begin{proof}
It is enough to show that \(q\) sends closed \(G\)-invariant subsets to closed subsets, and disjoint pairs of closed invariant subsets to disjoint pairs of closed subsets.

Firstly suppose that \(V\) is a  \(G\)-invariant Zariski-closed subset of \(X\). Then \(q(V) = q(V \cap \mu^{-1}(u))\), and \(V \cap \mu^{-1}(u)\) is \(K\)-invariant and closed in the classical topology of \(\mu^{-1}(u)\). This implies that \(q(V)\) is closed in the classical topology on \(\mu^{-1}(u)/K \simeq X\sslash_u G\). But \(q(V)\) is constructable, as the image of a Zariski-closed subset of \(X\), and so we may conclude that \(q(V)\) is Zariski-closed in \(\mu^{-1}(u)/K \simeq X\sslash_u G\).

Now suppose \(V,W\) are \(G\)-invariant and Zariski-closed in \(X\), with \(x \in V\) and \(y \in W\) such that \(q(x) = q(y)\). By \ref{thm:KN} we may take \({x'} \in \overline{Gx} \cap \mu^{-1}(u), \ {y'} \in \overline{Gy} \cap \mu^{-1}(u)\) such that \(q({x'}) = q({y'})\). These two points lie in the same \(K\)-orbit. By the \(G\)-invariance of \(V,W\) we have \(V \cap W \neq \emptyset\).
\end{proof}
%
%
%
%
%
%

\subsection{Log canonical thresholds} \

Here we recall the definition of the global log canonical threshold of a log pair, which features in Theorem~\ref{thm:SU} and Lemma~\ref{lem:alph}. A log pair \((Y,D)\) consists of a normal variety \(Y\) and a \(\QQ\)-divisor \(D\), where the coefficients of the irreducible components of \(D\) lie in \([0,1]\). The canonical divisor of such a pair is \(K_Y+D\). A pair \((Y,D)\) is called smooth if \(Y\) is smooth and \(D\) is a simple normal crossings divisor. A log resolution of a log pair \((Y,D)\) is a birational map \(\pi: \tilde{Y} \to Y\) such that \((\tilde{Y},\varphi^*D)\) is smooth.

\begin{definition}
Suppose \(\pi: \tilde{Y} \to Y\) is a log resolution of a pair \((Y,D)\). Write \(D = \sum a_i D_i\) for prime \(D_i\) and rational \(a_i\). Then:
\[
\pi^*(K_Y+ D) - K_{\tilde{Y}}\sim_{\QQ} \sum_i a_i \tilde{D_i} + \sum_j b_j E_j,
\]
where \(\tilde{D_i}\) is the proper transform of \(D_i\) and the \(E_j\) are the \(\pi\)-exceptional divisors. We say \((Y,D)\) is log canonical at \(P \in Y\) if we have \(a_i \le 1 \) for \(P \in D_i\), and \(b_j \le 1\) for \(E_j\) such that \(\pi(E_j) = P\). This condition is independent of the choice of resolution. If \((Y,D)\) is log canonical at all \(P  \in Y\) then we say \((Y,D)\) is (globally) log canonical.
\end{definition}
\begin{example} \label{examplelct}
Consider the pair \(Y = \PP^2\) and \(D = \sum a_i L_i\) where \(L_i\) are all lines through a point \(P \in Y\). Blowing up at \(P\) we obtain the following:
\[
\pi^*(K_Y+D) - K_{\tilde{Y}} \sim_{\QQ} (\deg D - 1) E + \sum a_i \tilde{L}_i,
\]
where \(E\) is the exceptional divisor  of the blow-up. Therefore \((Y,D)\) is log-canonical whenever we have \(\deg D \le 2\) and all \(a_i \le 1\).
\end{example}
Recall the following consequence of the main theorem of \cite{demailly2001}, as stated in the proof of \cite[Lemma 5.1]{cheltsov08}. This allows us to degenerate a pair under a \(\CC^*\)-action if we want to show it is log canonical.
\begin{proposition} \label{degenpair}
Let \((Y,D)\) be a log pair. Suppose \(\{ D_t | t \in \CC\}\) is a family of \(\QQ\)-divisors such that \(D_t \sim_\QQ D\), \(D_1 = D\), and for \(t \neq 0\) there exists \(\phi_t \in \Aut(X)\) such that \(D_t = \phi_t(D)\). Then \((Y,D)\) is log canonical if \((Y,D_0)\) is.
\end{proposition}
Now we recall the definition of the global log canonical threshold of a pair, as given in \cite{suess18-2}.
\begin{definition}
The global \(G\)-equivariant log canonical threshold of a log pair \((Y,B)\)  is defined to be
\[
\glct_{G}(Y,B) := \sup \{ \lambda | (Y,B+ \lambda D) \text{ log canonical } \forall D \in | -K_X - B |_{\QQ}^G \}.
\]
When \(B\) is trivial we will suppress it in our notation, writing \(\glct_G(X)\) for the \(G\)-equivariant log canonical threshold of a normal variety \(X\).
\end{definition}
In Demailly's appendix of \cite{cheltsov08} it is shown that \(\glct_{G}(X) = \alpha_G(X)\) for \(G \subseteq \Aut(X)\) a finite subgroup. The same proof may be easily extended to our setting, where \(G\) is the semidirect product of a torus \(T\) and a finite subgroup  \(H\) of the normalizer of \(T\) in \(\Aut(X)\). We outline one way of doing this in the following lemma.
\begin{lemma}
Suppose that \(X\) is a \(T\)-variety and \(H\) is a finite subgroup of the normalizer \( \mathcal{N}_{\Aut(X)}(T)\). Then \(\glct_{HT}(X) = \alpha_{HT}(X)\).
\end{lemma}
\begin{proof}
One may define the log canonical threshold of a linear system \(|\Sigma| \subset |mL|\) for any line bundle \(L\) on \(X\), see remarks succeeding \cite[Definition A.2]{cheltsov08}. Note, by definition, if \(D \in |\Sigma|\) then \(\lct(\frac{1}{m} D) \le \lct( \frac{1}{m} |\Sigma|)\) with equality when \(\Sigma\) is one-dimensional. As stated in \cite[(A.1)]{cheltsov08} we have:
\[
\alpha_{HT}(L) = \inf_{m \in \mathbb{Z}_{>0}} \inf_{\substack{|\Sigma| \subset |mL| \\ |\Sigma| \text{ is } HT-\text{invariant}}} \lct ( \frac{1}{m} |\Sigma|. ) 
\]
Clearly we have the inequality:
\[
\alpha_{HT}(L)  \le \inf_{m \in \mathbb{Z}_{>0}} \inf_{D \in |mL|^{HT}} \lct (\frac{1}{m} D).
\]
Now suppose \(|\Sigma| \subset |mL|\) is a \(HT\)-invariant linear system. Take \(D \in |\Sigma|\). We may repeatedly degenerate \(D\) along \(\CC^*\)-actions to obtain \(D' \in |mL|^{T}\), with \(\lct(\frac{1}{m}D') \le \lct(\frac{1}{m}D) \le \lct( \frac{1}{m}|\Sigma|)\) by Proposition~\ref{degenpair}. Let \(r = |H|\). Since \(H\) normalizes \(T\), we may take \(D'':= \sum_{h \in H} h \cdot D'\), and then:
\[
\lct(\frac{1}{m r }D'') \le \lct( \frac{1}{mr} | r\Sigma|) = \lct( \frac{1}{m}|\Sigma|).
\]
Then we have:
\[
\alpha_{HT}(L) = \inf_{m \in \mathbb{Z}_{>0}} \inf_{D \in |mL|^{HT}} \lct (\frac{1}{m} D).
\]
In particular when \(L = -K_X\) the left hand side is equal to \(\glct_{HT}(X)\).
\end{proof}
\section{Examples}
\subsection{Bidegree $(\alpha,\beta)$ hypersurfaces $X_{\alpha,\beta}^{2n-1}$} \label{subsec:hypersurfaces} \ \\

Recall the setup mentioned briefly in the introduction. Fix natural numbers \(n,\alpha,\beta>0\), and consider:
\[
X = X_{\alpha,\beta}^{2n-1} := V \left( \sum_{i=0}^n x_i^\alpha y_i^\beta \right) \subseteq \PP^n \times \PP^n.
\]
Let \(a = \alpha/d, b = \beta/d\), where \(d = \text{gcd}(\alpha,\beta)\) and let \(T\) be the \(n\)-torus acting with weights \((0|b I_n|0|-a I_n)\). Let \(K\) denote the maximal compact torus in \(T\).
First we calculate our GIT and Chow quotients, proving Lemma~\ref{lem:1.4}. Let \(L\) be the restriction of \(  \mathcal{O}(1,1)\) to \(X\). Using (\ref{eq:mu}) we can explicitely give a moment map for the torus action:
\[
([x],[y]) \mapsto \frac{ \sum |x_iy_j|^2( b e_i - a e_j)}{\sum |x_iy_j|^2}.
\]
Where we take \(e_0 := 0\). The moment image polytope \(P\) is the convex hull of the vectors \(\{ b e_i - a e_j \}_{i,j}\). Consider a boundary point \(u \in \partial P\). In this case we show that the moment fibre of \(u\) is contained in one \(T\)-orbit, and so the GIT quotient is just contraction to a point. The key observation here is the following:
\begin{lemma}\label{lem:X}
Suppose \(\mu([x],[y]) = \mu([x'],[y'])\) and for each \(j\) we have
\[
x_jy_j = {x'}_j{y'}_j = 0.
\]
Then for each \(j\) we have \(x_i=0 \iff {x'}_i = 0\) and \(y_i = 0 \iff {y'}_i = 0 \).
\end{lemma}
\begin{proof}
Suppose first \(i>0\). We have:
\[
A \sum_{j=0}^n |x_iy_j|^2 - B\sum_{j=0}^n  |x_jy_i|^2  = C \sum_{j=0}^n  |{x'}_i{y'}_j|^2 - D \sum_{j=0}^n  |{x'}_j{y'}_i|^2.
\] 
For positive constants \(A,B,C,D\). The conclusion follows by considering signs. Suppose now \(i =0\). By applying the affine linear functional \(l(w) := w \cdot \left( \sum e_j \right) - (b-a)\) to  \(\mu(x,y) = \mu(x',y')\) we obtain:
\[
 E \sum_{j=1}^n  |x_jy_0|^2  - F \sum_{j=0}^n |x_0y_j|^2 =  G \sum_{j=1}^n  |{x'}_j{y'}_0|^2 - H \sum_{j=0}^n  |{x'_0}{y'_j}|^2.
\]
For positive constants \(E,F,G,H\). Again by signs we obtain the result.
\end{proof}
\begin{lemma} \label{lem:3.2}
For \(u \in \partial P\) the moment fibre \(\mu^{-1}(u)\) is contained in one \(T\)-orbit.
\end{lemma}
\begin{proof}
Suppose \(\mu([x],[y]) = \mu([{x'}],[{y'}]) \in \partial P\). Since \( (\beta - \alpha)e_i \in P^\circ\) then \(x_iy_i = {x'}_i{y'}_i = 0\) for each \(i>0\). By the defining equation of \(X\) then also \(x_0y_0 = x_0y_0 = 0\). Applying Lemma~\ref{lem:X} we are done.
\end{proof}
Now consider moment fibres of points in the interior of \(P\). We calculate the associated GIT quotient by selecting an appropriate rational map, as in Lemma~\ref{lem:catquot}.
\begin{lemma} \label{lem:3.3}
For \(u \in P^\circ\) the topological quotient \(\mu^{-1}(u) \to \mu^{-1}(u)/K \) is:
\[
\mu^{-1}(u) \to \PP^{n-1}; \ \ ([x],[y]) \mapsto (x_1^a y_1^b: \dots : x_n^a y_n^b ).
\]
\end{lemma}
\begin{proof}
The map is clearly \(K\)-invariant. If \((x,y), ({x'},{y'}) \in \mu^{-1}(u)\) Then for any representatives \(x,y,{x'},{y'}\) we have:
\[
(x_1^a y_1^b: \dots : x_n^a y_n^b) = ({x'_1}^a {y'_1}^b: \dots : {x'_n}^a {y'_n}^b).
\]
Fix a representative \(x\) of \([x]\). Pick a representative \(y\) of \([y]\) such that \(|x||y| = 1\). For any representatives \({x'},{y'}\) of \([x'],[y']\) respectively, there is \(\lambda \in \CC^*\) such that \( x_i^a y_i^b = \lambda {{x'_i}}^a {{y'_i}}^b\) for \(i>0\). Now, by Lemma~\ref{lem:X} we may pick a representative \({x'}\) such that \(x_0 = {{x'_0}}\). Pick a representative \({y'}\) such that \(\lambda = 1\). \footnote{Note that rescaling our chosen \({y'}\) by an element of \(S^1\) does not change anything here.} By the defining equation of \(X\) then  \(x_0^\alpha y_0^\beta = {x'_0}^\alpha {y_0'}^\beta \). Applying Lemma~\ref{lem:X} we have \(\nu \in \CC^*\) such that \(y_0 = \nu{y'_0}\). As \(x_0 = {x'_0}\) then \(\nu \in S^1\), and we may rescale \({y'}\) by \(1/\nu\) so that \({y'_0} = y_0\).
\\ \\
If \(x_iy_i = 0\) then by Lemma~\ref{lem:X} we can pick \(t \in \CC^*\) such that \(x_i = t^\beta x_i', \ y_i = t^{-\alpha} x_i'\). Suppose now \(x_iy_i \neq 0\). Then \(x_i,y_i,{x'_i},{y'_i} \neq 0\).
Pick \(t\) such that \(t^a = {y'_i}/y_i\). Now \(x_i^a = t^{a b} {x'_i}^a\). Hence there exists some \(a\)th root of unity, say \(\xi\), such that \(x_i = \xi t^b {x'_i}\). Since \(a,b\) are coprime we may pick another \(a\)th root of unity \(\gamma\) such that \(\gamma^b = \xi\). Picking \(s\) such that \(s^d = \gamma t\) we obtain \(x_i = s^\beta {x'}_i\) and \(y_i = s^{-\alpha} {y'_i}\). We then have
\[
([x],[y]) \in T \cdot ([x'],[y']) \ \cap \ \mu^{-1}(u) = S \cdot ([x'],[y']).
\]
Thus we have described a closed map with fibres precisely the \(S\)-orbits of \(\mu^{-1}(u)\). This must be the topological quotient of the \(S\)-action.
\end{proof}
By Lemma~\ref{lem:catquot}, we see for any \(u \in P^\circ\) the GIT quotient, according to the associated linearization of \(L\), is given by:
\[
([x],[y]) \mapsto (x_1^a y_1^b: \dots : x_n^a y_n^b).
\]
By Lemma~\ref{lem:catquot} this implies the Chow quotient is given by the same formula. We may now calculate the boundary divisor of this quotient, completing the proof of Lemma~\ref{lem:1.4}.
\begin{proof}[Proof of Lemma~\ref{lem:1.4}]
From the above discussion we know the Chow quotient map is given by:
\[
X \to \PP^{n-1}; \ \ ([x],[y]) \mapsto (x_1^a y_1^b: \dots : x_n^a y_n^b ).
\]
Suppose that \(Z\) is a prime divisor on the quotient, and \(D\) is a component of \(q^{-1}(Z)\). If \(D\) intersects the open set where \(x_i,y_i \neq 0\), then \(t_i^a = t_i^b = 1\) for any \(t\) in the generic stabilizer of \(D\). As \(a,b\) are coprime this would imply \(t_i = 1\). Suppose \(D\) is a component of \(q^{-1}(Z)\) for some \(Z\) not of the form \(H_j\). Then for each \(i\), \(D\) intersects the open set where \(x_i,y_i \neq 0\), so \(D\) has trivial generic stabilizer.

Now consider the prime divisor \(H_j\) on the quotient, for some fixed \(j\). The irreducible components of \(q^{-1}(H_i)\) are given by the homogeneous ideals \((x_j^a)\), \((y_j^b)\). The generic stabilizer of the first is a cyclic group of order \(a\), generated by the element \(t \in T\) with \(t_i = 1\) for \(i \neq j\) and \(t_j\) a primitive \(a\)th root of unity. By symmetry the generic stabilizer of the second is a cyclic group of order \(b\). This gives the required boundary divisor.
\end{proof}
\begin{proof}[Proof of Corollary~\ref{cor:topquot}]
By Lemma~\ref{lem:3.2} and Lemma~\ref{lem:3.3} we see that the \(T\)-action on \(X^{2n-1}_{\alpha,\beta}\) has almost trivial variation of GIT, as defined in \cite[Definition 2.7]{suess18-2}, with \(Y = \PP^{n-1}\). The result follows by \cite[Proposition 2.9]{suess18-2}.
\end{proof}
\subsection{Wonderful compactification of the quadric $W$} \label{subsec:wonderful}
In this section we construct a wonderful compactification of an arrangement on an even dimensional quadric. We show that this compactification is Fano and we calculate the Chow quotient pair with respect to an induced torus action. First recall the notion of wonderful compactifications of arrangements of subvarieties, as introduced in \cite{li06}.
\begin{definition}
Let \(X\) be a nonsingular algebraic variety. An arrangement of subvarieties of \(X\) is a finite collection \(\mathcal{S}\) of subvarieties closed under pairwise scheme-theoretic intersection. A building set of \(\mathcal{S}\) is a subset \(\mathcal{G} \subset \mathcal{S}\) such that for any \(S \in \mathcal{S} \backslash \mathcal{G}\) the minimal elements of \(\{G \in \mathcal{G} | G \supset S\}\) intersect transversally and the intersection is \(S\). We will say that \(\mathcal{S}\) is built by \(\mathcal{G}\) if \(\mathcal{G}\) is a building set for \(\mathcal{S}\).
\end{definition}
We will use the following result:
\begin{theorem}[{\cite[Theorem 1.3]{li06}}] \label{thm:wonderful}
Let \(X\) be a nonsingular projective variety, and \(V_1,\dots,V_k\) a collection of subvarieties such any non-empty subset of \(\{V_1,\dots, V_k\}\) forms a building set for an arrangement of subvarieties. Consider the iterated blowup:
\[
W := \Bl_{\tilde{V}_k} \Bl_{\tilde{V}_{k-1}} \dots \Bl_{\tilde{V}_2} \Bl_{V_1} X.
\]
Then:
\begin{itemize}
\item each blowup is along a nonsingular variety;
\item \(W\) is isomorphic to the blowup along the ideal \(I_1 I_2 \cdots I_k\), where \(I_i\) is the homogeneous ideal corresponding to \(V_i\) for each \(i\).
\end{itemize}
\end{theorem}
Following \cite{li06}, we will call \(W\) the wonderful compactification of the arrangement built by \(V_1,\dots,V_k\). Note that the composition of blowups \(\pi: W \to X\) is independent of the ordering of the \(V_i\).

We now turn to our particular example mentioned in the introduction. Let \(W\) be the wonderful compactification of the arrangement of subvarieties of \(Q\) built by \(Z_0,\dots,Z_n\), where \(Z_i := V(x_{2i},x_{2i+1}) \subseteq Q\). We have the following:
\begin{lemma}
\(W\) is Fano.
\end{lemma}
\begin{proof}
By adjunction it is enough to show that \(-K_B - W\) is ample, where \(B\) is the wonderful compactification of the arrangement of subvarieties of \(\PP^{2n+1}\) built by \(V_0,\dots,V_n\), where \(V_i := V(x_{2i},x_{2i+1}) \subseteq \PP^{2n+1}\).

For each \(i\) pick \(\sigma_i \in S_n\) such that \(\sigma_i(1) = i\). Each \(\sigma_i\) corresponds to a sequence of blowups whose composition is independent of \(i\), as in Theorem~\ref{thm:wonderful}. Denote by \( \psi_i: \Bl_{V_i} \PP^{2n+1} \to \PP^{2n+1}\) the first blowup of this sequence, and \(\pi_i: B \to \Bl_{V_i}\) the composition of the remaining blowups, so that the wonderful compactification is given by the composition \(\pi_i \circ \psi_i\). Denote the exceptional divisor of \(\psi_i\) by \(E_i\).

Consider the divisor \(D_i := \psi_i^* \mathcal{O}(1) - E_i\) on \(\Bl_{V_i} \PP^{2n+1}\). Note that \(D_i\) is nef, since for any curve \(C\) in \(\Bl_{V_i} \PP^{2n+1}\) we may pick a hyperplane \(H \subset \PP^{2n+1} \) such that \(C \not\subset \tilde{H}\) but \(Z_i \subset \tilde{H}\), whereupon \((\pi^* \mathcal{O}(1) - E_i) \cdot C = \tilde{H} \cdot C \ge 0\).
Now
\[
-K_B - W \sim  \sum_{i=0}^{n} \pi_i^* D_i + (\pi_0 \circ \psi_0)^* \mathcal{O}(n).
\]
The divisors \((\pi \circ \psi)^* \mathcal{O}(1), \ \pi_0^* D_0, \dots, \pi_{n-1}^* D_{n} \) form a basis of the Picard group of \(W\). Therefore they span a full dimensional subcone of the nef cone of \(B\). Now \(-K_B - W\) is on the interior of this cone, and so is ample.
\end{proof}
By construction there is a natural morphism \(\pi: W \to Q\) which is a composition of blowups, each centered at a smooth subvariety by Theorem~\ref{thm:wonderful}. Fix the line bundle \(L = \mathcal{O}(1)_{|Q} \) on \(Q\). Recall that there is an \(n\)-torus \(T\) acting on \(Q\) prescribed by \(\deg x_{2i} = e_{i+1}, \deg x_{2i+1} = -e_{i+1}\). This torus action may be extended to \(W\). These torus actions are not effective, but we may quotient by the global stabilizer, a cyclic group of order  two generated by \(-\text{Id} = (-1,\dots,-1) \in T\), to obtain the action of an effective torus \(T'\) on \(Q\) and \(W\). Quotienting does not affect the calculation of GIT quotients. In \cite{suess18-2} the GIT quotients \(q: Q \to Q \sslash T'\) were determined. They are either trivial contractions to a point, or of the following form:
\begin{equation} \label{blowupquot}
Q \to \PP^{n-1}; \ \ [x] \mapsto (x_1x_2:\dots :x_{2n-1}x_{2n}).
\end{equation}
The Chow quotient is also then given by (\ref{blowupquot}). Following \cite{kirwan}, there is an ample line bundle \(\tilde{L}\) on \(W\) such that any linearization of \(\tilde{L}\) is a lift of a linearization of \(L\). Moreover, given a linearization, it can  be shown that \(W^{ss} \subset \pi^{-1}(X^{ss})\). By \cite[Lemma 3.11]{kirwan} the GIT quotients of \(W\) given by a linearization of \(\tilde{L}\) are precisely the restrictions of compositions \(q \circ \pi\), where \(q\) is the GIT quotient map for \(Q\) given by the corresponding linearization of \(L\).

We can conclude that the Chow quotient is the restriction of a composition of the blowup map \(\pi\) followed by the map (\ref{blowupquot}). We now calculate the boundary divisor of this quotient, proving Lemma~\ref{lem:1.6}:
\begin{proof}[Proof of Lemma~\ref{lem:1.6}]
Recall the definition of the boundary divisor in the Chow quotient pair, see (\ref{boundary}). From the formula (\ref{blowupquot}) it is easy to calculate that the boundary divisor of the Chow quotient pair of \(Q\) is trivial. Therefore the only chance for \(m_Z >1\) occurs at the exceptional loci of blowups. If we construct \(W\) with the following sequence of blowups
\[
W = \Bl_{\tilde{Z}_{n}} \dots \Bl_{\tilde{Z}_1} \Bl_{Z_0} Q.
\]
The exceptional divisor of the composition of blowup maps is of the form \(E_{n-1} + \dots + E_0\), where \(E_i\) is the exceptional divisor of the \((i+1)\)th blowup in the sequence. By symmetry it is enough to calculate the generic stabilizer of \(E_0\). Consider \(\Bl_{Z_0}Q\), realized as a subvariety of \(Q \times \PP^1 \), given by the additional equation \(vx_0 - ux_1\), where \(u,v\) are the homogeneous variables in the second factor.

There is an induced \(T\)-action on \(\Bl_{Z_0}\), under which the equation \(vx_0 - ux_1\) must be homogeneous with respect to the induced grading of the character lattice of \(T\). This implies that \(\deg u = \deg v + 2e_1\), and we see that the generic stabilizer of the \(T' = T/ \langle \pm \text{Id} \rangle \)-action on the exceptional divisor must be a cyclic group of order \(2\), generated by the element \((-1,1,\dots,1) + \langle \pm \text{Id} \rangle\). 
\end{proof}
\section{Threshold Bounds and new K\"ahler-Einstein metrics}
Let \(Y = \PP^2\) with projective coordinates \(x_1,x_2,x_3\) and consider the boundary divisor \(B_\gamma  := \gamma \sum_{i=0}^3  H_i\), where \(H_1,H_2,H_3\) are the coordinate hyperplanes, and \(H_0 = V (\sum_i x_i) \). Consider the subgroup \(G \cong S_4\) of \(\text{Aut}(Y)\) permuting the hyperplanes \(H_0,\dots,H_3\). First we prove Lemma~\ref{lem:alph}:
\begin{proof}[Proof of Lemma~\ref{lem:alph}] \

To show \(\glct_{G}(Y,B_\gamma) \ge \lambda\) it is sufficient to show that \(B_\gamma+\lambda D\) is log canonical for any \( D \in |-K_Y - B_\gamma|_{\QQ}^G\). Fix such a \(D\) and take \(P \in Y\). At most two of the \(H_i\) pass through \(P\), so without loss of generality suppose \(H_0,H_3\) do not. Modify \(B_\gamma+\lambda D\) by removing any components supported at \(H_0,H_3\), to obtain a divisor \(D'\). Note \(B_\gamma+\lambda D\) is log canonical at \(P\) if \(D'\) is globally log canonical. Note also that although \(D'\) may not be \(G\)-invariant, it is still invariant under the involution \(\sigma \) swapping \(x_1\) and \(x_2\). Finally note \(D' \ge \gamma H_1 + \gamma H_2\).

Consider the \(\CC^*\)-action \(t \cdot [x_1:x_2:x_3] = [ t x_1: t x_2 : x_3]\). By \ref{degenpair}, \(D'\) is log canonical if \(D_0' := \lim_{t \to 0} \left( t \cdot D' \right) \) is log canonical. This \(\CC^*\)-action commutes with \(\sigma \), and so \(D_0'\) is invariant under \(\sigma\). Moreover is is clear that each component of \(D_0'\) must be a line through the point \([0,0,1]\). By \(\sigma\)-invariance \(D_0'\) must be of the form:
\[
D_0' = \gamma (H_1+H_2) + a V (x_1+x_2) + b V(x_1-x_2) + \sum_i c_i (L_i + \sigma L_i),
\]
where \(a + b + \sum_i 2 c_i \le 2 \gamma + 3\lambda - 4 \gamma \lambda\). This is a divisor of the form described in Example \ref{examplelct}. It is therefore log canonical if the following inequalities hold:
\begin{align*}
2\gamma + 3 \lambda - 4\gamma \lambda &< 2; \\
3\gamma - 4\gamma \lambda &\le 1.
\end{align*}
Basic manipulation of inequalities gives our bounds on the global threshold.
\end{proof}
\begin{proof}[Proof of Theorems \ref{thm:KE1} and \ref{thm:KE2}]
If \(X\) is \(X_{1,2}^5\) or \(W^6\) then \(X\) has Chow quotient \((\PP^2,B_{1/2})\), and applying Theorem~\ref{thm:SU} and Lemma~\ref{lem:alph} we see that \(\alpha_{S_4}(X) \ge 1\). By Tian's criterion then \(X\) admits an invariant K\"ahler-Einstein metric. Similarly the Chow quotient of \(X_{1,3}^5\) is \((\PP^2,B_{2/3})\), and we see \(\alpha_{S_4}(X_{1,3}^5) \ge 2\), so \(X_{1,3}^5\) is also K\"ahler-Einstein.
\end{proof}
\bibliography{biblio} 

\begin{thebibliography}{10}

\bibitem{CDS12}
X.-X. Chen, S.~Donaldson, and S.~Sun, ``K{\"a}hler-{E}instein metrics and
  stability,'' {\em Int. Math.Res. Not. IMRN}, no.~8, pp.~2119--2125, 2014.

\bibitem{CDS13}
X.~Chen, S.~Donaldson, and S.~Sun, ``K{\"a}hler-{E}instein metrics on {F}ano
  manifolds. {I}{I}{I}: {L}imits as cone angle approaches $2 \pi$ and
  completion of the main proof,'' {\em Journal of the American Mathematical
  Society}, vol.~28, no.~1, pp.~235--278, 2015.

\bibitem{del2016}
T.~Delcroix, ``K-stability of {F}ano spherical varieties,'' {\em arXiv preprint
  arXiv:1608.01852}, 2016.

\bibitem{ilten2015}
N.~Ilten and H.~S{\"u}{\ss}, ``K-stability for {F}ano manifolds with torus
  action of complexity 1,'' {\em Duke Math. J.}, vol.~166, no.~1, pp.~177--204,
  2017.

\bibitem{Tian87}
G.~Tian, ``On {K}{\"a}hler-{E}instein metrics on certain {K}{\"a}hler manifolds
  with $c_1 (m)> 0$,'' {\em Inventiones mathematicae}, vol.~89, no.~2,
  pp.~225--246, 1987.

\bibitem{Su13}
H.~S{\"u}{\ss}, ``{K}{\"a}hler-{E}instein metrics on symmetric {Fano}
  {$T$-varieties},'' {\em Advances in Mathematics}, vol.~246, pp.~100 -- 113,
  2013.

\bibitem{hausen2018torus}
J.~Hausen, C.~Hische, and M.~Wrobel, ``On torus actions of higher complexity,''
  {\em arXiv preprint arXiv:1802.00417}, 2018.

\bibitem{Matsushima}
Y.~{M}atsushima, ``Remarks on {K}{\"a}hler-{E}instein manifolds,'' {\em Nagoya
  {M}ath. {J}.}, vol.~46, pp.~161--173, 1972.

\bibitem{cheltsov08}
I.~A. Cheltsov and K.~A. Shramov, ``Log canonical thresholds of smooth {F}ano
  threefolds,'' {\em Russian Mathematical Surveys}, vol.~63, no.~5, p.~859,
  2008.

\bibitem{batyrev99}
V.~V. Batyrev and E.~N. Selivanova, ``Einstein-{K}{\"a}hler metrics on
  symmetric toric {F}ano manifolds,'' {\em {J}. {R}eine {A}ngew. {M}ath.},
  no.~512, pp.~225 -- 236, 1999.

\bibitem{kapranov1993}
M.~M. Kapranov, ``Chow quotients of {G}rassmannians. {I},'' {\em Adv. Soviet
  Math}, vol.~16, no.~2, pp.~29--110, 1993.

\bibitem{baker2012}
H.~B{\"a}ker, J.~Hausen, and S.~Keicher, ``On {C}how quotients of torus
  actions,'' {\em Michigan Math. J.}, vol.~64, no.~3, pp.~451--473, 2015.

\bibitem{suess18-2}
H.~S{\"u}{\ss}, ``{Orbit spaces of maximal torus actions on oriented
  {G}rassmannians of planes}.'' 2018.

\bibitem{kirwan}
F.~C. Kirwan, ``Partial desingularisations of quotients of nonsingular
  varieties and their {B}etti numbers,'' {\em Annals of mathematics}, vol.~122,
  no.~1, pp.~41--85, 1985.

\bibitem{li06}
L.~Li, ``Wonderful compactification of an arrangement of subvarieties,'' {\em
  Michigan Math. J.}, vol.~58, no.~2, pp.~535--563, 2009.

\bibitem{mumford1994}
D.~Mumford, J.~Fogarty, and F.~Kirwan, {\em Geometric invariant theory},
  vol.~34.
\newblock Springer Science \& Business Media, 1994.

\bibitem{atiyah1982convexity}
M.~F. Atiyah, ``Convexity and commuting {H}amiltonians,'' {\em Bulletin of the
  London Mathematical Society}, vol.~14, no.~1, pp.~1--15, 1982.

\bibitem{guillemin1982convexity}
V.~Guillemin and S.~Sternberg, ``Convexity properties of the moment mapping,''
  {\em Inventiones mathematicae}, vol.~67, no.~3, pp.~491--513, 1982.

\bibitem{kempf1979}
G.~Kempf and L.~Ness, ``The length of vectors in representation spaces,'' in
  {\em Algebraic geometry}, pp.~233--243, Springer, 1979.

\bibitem{demailly2001}
J.-P. Demailly and J.~Koll{\'a}r, ``Semi-continuity of complex singularity
  exponents and {K}{\"a}hler-{E}instein metrics on {F}ano orbifolds,'' in {\em
  Annales scientifiques de l'Ecole normale sup{\'e}rieure}, vol.~34,
  pp.~525--556, 2001.

\end{thebibliography}
\bibliographystyle{ieeetr}
\end{document}